\documentclass[reqno]{amsart}

\usepackage{amsfonts,amsthm,amsmath,amssymb,amscd, centernot}
\usepackage{esint}
\usepackage{mathrsfs}
\usepackage{mathtools}
\usepackage{graphics}
\usepackage{indentfirst}
\usepackage{cite}
\usepackage{latexsym}
\usepackage[dvips]{epsfig}
\usepackage{color}
\usepackage{hyperref}

\newtheorem{Theorem}{Theorem}
\newtheorem*{Theorem*}{Theorem}

\newtheorem*{Corollary*}{Corollary}
\newtheorem{Lemma}{Lemma}
\newtheorem*{Lemma*}{Lemma}
\newtheorem{Proposition}{Proposition}
\newtheorem*{Proposition*}{Proposition}

\newtheorem*{Definition*}{Definition}
\newtheorem{Remark}{Remark}
\newtheorem*{Remark*}{Remark}

\newtheorem*{Example*}{Example}

\newcommand{\ptl}{\partial}

\DeclareMathOperator{\supp}{supp}

\title[3D Hyperviscous Navier-Stokes Equations]{Non-uniqueness of Weak Solutions to Hyperviscous Navier-Stokes Equations - On Sharpness of J.-L. Lions Exponent}

\author[]{Tianwen Luo$^*$}
\thanks{$^*$Yau Mathematical Sciences Center, Tsinghua University, China. twluo@mail.tsinghua.edu.cn}

\author[]{Edriss S. Titi$^\dagger$}
\thanks{$^\dagger$Department of Mathematics, Texas A\&M University, 3368 TAMU, College Station, TX 77843-3368,	USA. Department of Applied Mathematics and Theoretical Physics, University of Cambridge, Wilberforce Road, Cambridge CB3 0WA, UK. Department of Computer Science and Applied Mathematics, The Weizmann Institute of Science, Rehovot 76100, Israel. titi@math.tamu.edu;  Edriss.Titi@damtp.cam.ac.uk; edriss.titi@weizmann.ac.il}

\date{January 13, 2020}

\begin{document}
	
\begin{abstract}
	Using the convex integration technique for the three-dimensional Navier-Stokes equations introduced by T. Buckmaster and V. Vicol, it is shown the existence of non-unique weak solutions for the 3D Navier-Stokes equations with fractional hyperviscosity $(-\Delta)^{\theta}$, whenever the exponent $\theta$ is less than J.-L. Lions' exponent $5/4$, i.e., when $\theta < 5/4$.	
\end{abstract}


\maketitle

\section{Introduction}

In this paper we consider the question of non-uniquness of weak solutions to the
 3D Navier-Stokes equations with fractional viscosity (FVNSE) on $\mathbb{T}^3$
\begin{align}\label{eq:FVNSE}
\begin{cases}
\ptl_t v + \nabla \cdot (v \otimes v) +  \nabla p + \nu(-  \Delta)^{\theta} v = 0,\\
\nabla \cdot v = 0,
\end{cases}
\end{align}
where $\theta \in \mathbb{R}$ is a fixed constant, and for $u \in C^{\infty}(\mathbb{T}^3)$ with $\int_{\mathbb{T}^3} u(x) dx =0$,  the fractional Laplacian is defined via the Fourier transform as
\begin{align*}
\mathcal{F}((-  \Delta)^{\theta} u)(\xi) = |\xi|^{2\theta}\mathcal{F}(u)(\xi), \quad \xi \in \mathbb{Z}^3.
\end{align*}
\begin{Definition*}[weak solutions]
	A vector field $v \in C^0_{weak}(\mathbb{R};L^2(\mathbb{T}^3))$ is called a weak solution to the FVNSE if it solves \eqref{eq:FVNSE} in the sense of distribution.
\end{Definition*}

When $\theta = 1$, FVNSE \eqref{eq:FVNSE} is the standard Navier-Stokes equations.
J.-L. Lions  first considered FVNSE \eqref{eq:FVNSE} in \cite{Lions59}, and showed the existence and uniqueness of weak solutions to the initial value problem, which also satisfied the energy equality, for $\theta \in [5/4,\infty)$ in \cite{Lions69}. Moreover, an analogue of the Caffarelli-Kohn-Nirenberg \cite{CKN} result was established in \cite{KatzPavlovic} for the FVNSE system \eqref{eq:FVNSE}, showing that the Hausdorff dimension of the singular set, in space and time, is bounded by $5 - 4\theta$ for $\theta \in (1,5/4)$. The existence, uniqueness, regularity and stability of solutions to the FVNSE have been studied in \cite{OlsonTiti05,JiuWang14,Wu03,Tao09} and references therein. Very recently, using the method of convex integration introduced in \cite{dLSz4}, Colombo, De Lellis and De Rosa in \cite{CdLdR18} showed the non-uniquenss of Leray weak solutions to  FVNSE \eqref{eq:FVNSE}  for $\theta \in (0,1/5)$ and for $\theta \in (0,1/3)$ in \cite{DeRosa19}.

In the recent breakthrough work \cite{BV17}, Buckmaster and Vicol obtained non-uniqueness of weak solutions to the three-dimensional Navier-Stokes equations. They developed a new convex integration scheme in Sobolev spaces using intermittent Beltrami flows which combined concentrations and oscillations. Later, the idea of using intermittent flows was used to study non-uniqueness for transport equations in \cite{MS17,MS19,Modena-Sattig-19} employing scaled Mikado waves, and for stationary Navier-Stokes equations in \cite{Luo19,Cheskidov-Luo-19} employing viscous eddies.

The schemes in \cite{BV17,MS17} are based on the convex integration framework in H\"{o}lder spaces for the Euler equations, introduced by De Lellis and Sz{\'e}kelyhidi in \cite{dLSz4}, subsequently refined in \cite{Isett12,Buckmaster2013transporting,Buckmaster2014,DaneriSzekelyhidi16}, and culminated in the proof of the second half of the Onsager conjecture by Isett in \cite{Isett16}; also see \cite{BDSV17} for a shorter proof. For the first half of the Onsager conjecture, see, e.g.,  \cite{CET94,BTi}, and the references therein.

The main contribution of this note is to show that the results in Buckmaster-Vicol's paper hold for FVNSE \eqref{eq:FVNSE} for $\theta < 5/4$:

\begin{Theorem}\label{Thm:Main}
	Assume that $\theta \in [1, 5/4)$.
	Suppose $u$ is a smooth divergence-free vector field, define on $\mathbb{R}_+ \times \mathbb{T}^3$, with compact support in time and satisfies the condition
	\begin{align*}
	\int_{\mathbb{T}^3} u(t,x)dx \equiv 0.
	\end{align*}
	Then for any given $\varepsilon_0 > 0$, there exists a weak solution $v$ to the FVNSE \eqref{eq:FVNSE}, with compact support in time, satisfying
	\begin{align*}
	\|v - u\|_{L^{\infty}_t W^{2\theta - 1,1}_x} < \varepsilon_0.
	\end{align*}
	As a consequence there are infinitely many weak solutions of the FVNSE \eqref{eq:FVNSE} which are compactly supported in time; in particular, there are infinitely many weak solutions with initial values zero.
\end{Theorem}

\begin{Remark}
	In the above theorem we assume that $\theta \in [1, 5/4)$. However, using the constructions in \cite{BV17} with a slightly different choice of parameters, one can actually show that Theorem 1.2 and Theorem 1.3 in \cite{BV17} hold for the 3D FVNSE, i.e., there exist non-unique weak solutions $v \in C_t^0 W_x^{\beta,2}$, with a different $\beta > 0$, depending on $\theta$. However, in this paper we choose to prove a weaker result, Theorem \ref{Thm:Main}, in order to simplify the presentation while retaining the main idea.
\end{Remark}

\begin{Remark}
	For the case $\theta \in (-\infty,1)$, the same construction also yields weak solutions $v \in C^0_t L^2_x \cap C^0_t W^{1,1}_x$ with a suitable choice of parameters.
\end{Remark}

We now make some comments on the analysis in this paper. Using the technique in \cite{BV17}, we adapt a convex integration scheme with intermittent Beltrami flows as the building blocks. The main difficulty in a convex integration scheme for  (FVNSE), is the error induced by the frictional viscosity $\nu(-  \Delta)^{\theta} v$, which is greater for a larger exponent $\theta$. This error is controlled by making full use of the concentration effect of intermittent flows introduced in  \cite{BV17}. As it is shown in the crucial estimate \eqref{est-R-linear}, the error is controllable only for $\theta < 5/4$. Compared with \cite{BV17}, since our goal is to construct weak solutions $v \in C^0_t L^2_{x,weak} \cap L^{\infty}_t W^{2\theta - 1,1}_x$, we adapt a slightly simpler cut-off function and prove only estimates that are sufficient for this purpose.


\section{Outline}

\subsection{Iteration lemma}
Following \cite{BV17}, we
consider the approximate system
\begin{align}\label{eq:NS-alpha-Reynold-stress}
\begin{cases}
\ptl_t v + \nabla \cdot (v \otimes v) +  \nabla p + \nu(-  \Delta)^{\theta} v = \nabla \cdot R,\\
\nabla \cdot v = 0,
\end{cases}
\end{align}
where $R$ is a symmetric $3 \times 3$ matrix.

\begin{Lemma}[Iteration Lemma for $L^2$ weak solutions]\label{Lemma:Iteration}
	Let $\theta \in (-\infty, 5/4)$.
	Assume $(v_q, R_q)$ is a smooth solution to \eqref{eq:NS-alpha-Reynold-stress} with
	\begin{align}
	\|R_q\|_{L^{\infty}_t L^1_x} &\leq \delta_{q+1}, \label{est-R_q-L^1}
	\end{align}
	for some $\delta_{q+1} > 0$.
	Then for any given $\delta_{q+2} > 0$, there exists a smooth solution $(v_{q+1}, R_{q+1})$ of \eqref{eq:NS-alpha-Reynold-stress}
	with
	\begin{align}
	\|R_{q+1}\|_{L^{\infty}_t L^1_x} &\leq \delta_{q+2}, \label{est-R-q+1}\\
	\text{and}\quad  \operatorname{supp}_t v_{q+1} \cup \operatorname{supp}_t R_{q+1}  &\subset N_{\delta_{q+1}}(\operatorname{supp}_t v_{q} \cup \operatorname{supp}_t R_{q}). \label{est-supp-v_q-R_q}
	\end{align}
	Here for a given set $A \subset \mathbb{R}$,  the $\delta$-neighborhood of $A$ is denoted by
	\begin{align*}
	N_{\delta}(A) = \{ y \in \mathbb{R}: \exists y' \in A, |y-y'| < \delta \}.
	\end{align*}
	Furthermore, the increment $w_{q+1} = v_{q+1} - v_q$ satisfies the estimates
	\begin{align}
	\|w_{q+1}\|_{L^{\infty}_t L^2_x} &\leq C \delta_{q+1}^{1/2}, \label{est-w_q-L^2_x}\\
	\|w_{q+1}\|_{L^{\infty}_t W^{2 \theta - 1,1}_x} &\leq \delta_{q+2}, \label{est-w_q-W^1_x}
	\end{align}
	where the positive constant $C$ depends only on $\theta$.
\end{Lemma}

\begin{proof}[Proof of Theorem \ref{Thm:Main}]
	Assume Lemma \ref{Lemma:Iteration} is valid. Let $v_0 = u$. Then
	\begin{align*}
	\int_{\mathbb{T}^3} \ptl_t v_0(t,x)dx  = \frac{d}{dt} \int_{\mathbb{T}^3} v_0(t,x)dx \equiv 0.
	\end{align*}
	Let
	\begin{align*}
	R_0 = \mathcal{R}(\ptl_t v_0 + \nu(-  \Delta)^{\theta} v_0)  + v_0 \otimes v_0 + p_0 I , \quad p_0 = -\frac{1}{3} |v_0|^2,
	\end{align*}
	where $\mathcal{R}$ is the symmetric anti-divergence operator established in Lemma \ref{Lemma:symm-anti-div}, below.
	Clearly $(v_0,R_0)$ solves \eqref{eq:NS-alpha-Reynold-stress}.
	Set
	\begin{align*}
	\delta_1 &= \|R_0\|_{L^{\infty}_t L^1_x}, \\
	\delta_{q+1} &= 2^{-q} \varepsilon_0, \quad \text{ for } q \geq 1.
	\end{align*}
	Apply Lemma \ref{Lemma:Iteration} iteratively to obtain smooth solution $(v_q, R_q)$ to \eqref{eq:NS-alpha-Reynold-stress}. It follows from \eqref{est-w_q-L^2_x} that
	\begin{align*}
	\sum \|v_{q+1} - v_{q}\|_{L^{\infty}_t L^2_x} = \sum \|w_{q+1}\|_{L^{\infty}_t L^2_x} \leq C \sum \delta_{q+1}^{1/2} < \infty.
	\end{align*}
	Thus $v_q$ converge strongly to some $v \in C^0_t L^2_x$. Since $\|R_{q+1}\|_{L^{\infty}_t L^1_x} \to 0$, as $q \to \infty$,  $v$ is  a weak solution to the FVNSE \eqref{eq:FVNSE}. Estimate \eqref{est-w_q-W^1_x} leads to
	\begin{align*}
	\|v - v_0\|_{_{L^{\infty}_t W^{2\theta - 1,1}_x}} \leq \sum_{q=1}^{\infty} \|w_q\|_{_{L^{\infty}_t W^{2\theta - 1,1}_x}} \leq \sum_{q=1}^{\infty}\delta_{q+1} \leq \varepsilon_0.
	\end{align*}
	Furthermore, it follows from \eqref{est-supp-v_q-R_q} that
	\begin{align*}
	\operatorname{supp}_t v  &\subset \cup_{q \geq 0} \operatorname{supp}_t v_q  \subset  N_{\sum_{q \geq 0} \delta_{q+1}}(\operatorname{supp}_t u)  \subset   N_{\delta_1 + \varepsilon_0}(\operatorname{supp}_t u).
	\end{align*}
	
	Now we show the existence of infinitely many weak solutions with initial values zero. Let $u(t,x) = \varphi(t) \sum_{|k| \leq N} a_k e^{ik \cdot x}$ with $a_k \neq 0, a_k \cdot k = 0, a_{-k} = a_k^*$ for all $|k| \leq N$, and $\varphi \in C_c^{\infty}(\mathbb{R}_+)$. Thus $\nabla \cdot u = 0$ satisfies the conditions of the theorem. Hence there exists a weak solution $v$ to \eqref{eq:FVNSE} close enough to $u$ so that $v \centernot{\equiv} 0$.

\end{proof}

\section{Iteration scheme}

\subsection{Notations and Parameters}
For a complex number $\zeta \in \mathbb{C}$, we denote by $\zeta^*$ its complex conjugate. Let us normalize the volume
\begin{align*}
|\mathbb{T}^3| = 1.
\end{align*}
For smooth functions $u \in C^{\infty}(\mathbb{T}^3)$ with $\int_{\mathbb{T}^3} u(x) dx =0$ and $s \in \mathbb{R}$, we define
\begin{align*}
\mathcal{F}(|\nabla|^s u)(\xi) = |\xi|^{s}\mathcal{F}(u)(\xi), \quad \xi \in \mathbb{Z}^3.
\end{align*}
For $M, N \in [0,+\infty]$, denote the Fourier projection of $u$ by
\begin{align*}
	\mathcal{F} (\mathbb{P}_{[M,N)} u) = \begin{cases}
	u(\xi), & M \leq |\xi| < N, \xi \in \mathbb{Z}^3,\\
	0,  &\text{otherwise}.
	\end{cases}
\end{align*}
We also denote $\mathbb{P}_{\leq k} = \mathbb{P}_{[0,k)}$ and $\mathbb{P}_{\geq k} = \mathbb{P}_{[k,+\infty)}$ for $k > 0$.

Following the notation in \cite{BV17}, we introduce here several parameters $\sigma, r, \lambda$, with
\begin{align}
0 <  \sigma   < 1 < r < \lambda < \mu < \lambda^2, \quad \sigma r < 1, \label{ineq:parameters}
\end{align}
where $\lambda = \lambda_{q+1} \in  5\mathbb{N}$ is the `frequency' parameter; $\sigma$ with $1/\sigma \in \mathbb{N}$ is a small parameter such that $\lambda \sigma \in \mathbb{N}$ parameterizes the spacing between frequencies; $r \in \mathbb{N}$ denotes the number of frequencies along edges of a cube; $\mu$ measures the amount of temporal oscillation.

Later $\sigma, r, \mu$ will be chosen to be suitable powers of $\lambda_{q+1}$. We also fix a constant $p > 1$ which will be chosen later to be close to $1$. The constants implicitly in the notation `$\lesssim$' may depend on $p$ but are independent of the parameters $\sigma, r, \lambda$.

\subsection{Intermittent Beltrami flows}

We use intermittent Beltrami flows introduced in \cite{BV17} as the building blocks.
Recall some basic facts of Beltrami waves.
\begin{Proposition}\label{Prop:Beltriami-Waves}(\cite[Proposition 3.1]{BV17})
	Given $\overline{\xi} \in \mathbb{S}^2 \cap \mathbb{Q}^3$, let $A_{\overline{\xi}} \in \mathbb{S}^2 \cap \mathbb{Q}^3$ be such that
	\begin{align*}
	A_{\overline{\xi}} \cdot \overline{\xi} = 0, \quad |A_{\overline{\xi}}| = 1, \quad A_{-\overline{\xi}} = A_{\overline{\xi}}.
	\end{align*}
	Let $\Lambda$ be a given finite subset of $\mathbb{S}^2$ such that $- \Lambda = \Lambda$, and $\lambda \in \mathbb{Z}$ be such that $\lambda \Lambda \subset \mathbb{Z}^3$. Then for any choice of coefficients $a_{\overline{\xi}} \in \mathbb{C}$ with $a_{\overline{\xi}}^* = a_{-\overline{\xi}}$ the vector field
	\begin{align*}
	W(x) = \sum_{\overline{\xi} \in \Lambda} a_{\overline{\xi}} B_{\overline{\xi}} e^{i \lambda \overline{\xi} \cdot x}, \quad \text{ with } B_{\overline{\xi}} = \frac{1}{\sqrt{2}}(A_{\overline{\xi}} + i \overline{\xi} \times A_{\overline{\xi}}),
	\end{align*}
	is real-valued, divergence-free and satisfies
	\begin{align*}
	\nabla \times W = \lambda W, \quad
	\nabla \cdot (W \otimes W) = \nabla \frac{|W|^2}{2}.
	\end{align*}
	Furthermore,
	\begin{align*}
	\langle W \otimes W \rangle := \fint_{ \mathbb{T}^3} W \otimes W dx= \sum_{\overline{\xi} \in \Lambda} \frac{1}{2}|a_{(\overline{\xi})}|^2(\mathrm{Id} - \overline{\xi} \otimes \overline{\xi}).
	\end{align*}
\end{Proposition}

Let $\Lambda, \Lambda^+, \Lambda^- \subset \mathbb{S}^2 \cap \mathbb{Q}^3$ be defined by
\begin{align*}
\Lambda^+ = \{ \frac{1}{5}(3e_1 \pm 4e_2), \frac{1}{5}(3e_2 \pm 4e_3), \frac{1}{5}(3e_3 \pm 4e_1) \},\\
\quad \Lambda^- = -\Lambda^+, \quad
\Lambda = \Lambda^+ \cup \Lambda^-.
\end{align*}
Clearly we have
\begin{align}
5\Lambda \in \mathbb{Z}^3, \quad \text{ and } \quad \min_{\overline{\xi}', \overline{\xi} \in \Lambda, \overline{\xi}'+ \overline{\xi}\neq 0} |\overline{\xi}'+ \overline{\xi}| \geq \frac{1}{5}. \label{est-xi+xi'}
\end{align}
Also it is direct to check that
\begin{align*}
\frac{1}{8}\sum_{\overline{\xi} \in \Lambda}(\mathrm{Id} - \overline{\xi} \otimes \overline{\xi}) = \mathrm{Id}.
\end{align*}
In fact, representations of this form exist for symmetric matrices close to the identity. We have the following simple variant of \cite[Proposition 3.2]{BV17}.
\begin{Proposition}\label{Prop:convex-representation}
	Let $B_{\varepsilon}(\mathrm{Id})$ denote the ball of symmetric matrices, centered at the identity, of radius $\varepsilon$. Then
	there exist a constant $\varepsilon_{\gamma} > 0$
	and smooth positive functions $\gamma_{(\overline{\xi})} \in C^{\infty}(B_{\varepsilon_{\gamma}}(\mathrm{Id}))$,
	such that
	\begin{enumerate}
		\item $\gamma_{(\overline{\xi})} = \gamma_{(-\overline{\xi})}$;
		\item for each $R \in B_{\varepsilon_{\gamma}}(\mathrm{Id})$ we have the identity
		\begin{align*}
		R = \frac{1}{2}\sum_{\overline{\xi} \in \Lambda} \left(\gamma_{(\overline{\xi})}(R)\right)^2(\mathrm{Id} - \overline{\xi} \otimes \overline{\xi}).
		\end{align*}
	\end{enumerate}
\end{Proposition}

Define the Dirichlet kernel
\begin{align*}
D_r(x) &= \frac{1}{(2r+1)^{3/2}} \sum_{\xi \in \Omega_r} e^{i \xi \cdot x}, \quad
\Omega_r = \{(j,k,l): j,k,l \in \{-r,\cdots,r\} \}.
\end{align*}
It has the property that, for $1 < p \leq \infty$,
\begin{align*}
\|D_r\|_{L^p} \lesssim r^{3/2 - 3/p}, \quad \|D_r\|_{L^2} = (2 \pi)^3.
\end{align*}

Following \cite{BV17}, for $\overline{\xi} \in \Lambda^+$, define a directed and rescaled Dirichlet kernel by
\begin{align}
\eta_{(\overline{\xi})}(t,x) = \eta_{\overline{\xi}, \lambda, \sigma, r, \mu}(t,x) = D_r(\lambda \sigma (\overline{\xi} \cdot x + \mu t, A_{\overline{\xi}} \cdot x, (\overline{\xi} \times A_{\overline{\xi}}) \cdot x)),
\end{align}
and for $\overline{\xi} \in \Lambda^-$, define
\begin{align*}
\eta_{(\overline{\xi})}(t,x) = \eta_{-(\overline{\xi})}(t,x).
\end{align*}
Note the important identity
\begin{align}
\frac{1}{\mu} \ptl_t \eta_{(\overline{\xi})}(t,x) = \pm (\overline{\xi} \cdot \nabla) \eta_{(\overline{\xi})}(t,x), \quad \overline{\xi} \in  \Lambda^{\pm}. \label{eq:Dt-eta}
\end{align}
Since the map $x \mapsto \lambda \sigma (\overline{\xi} \cdot x + \mu t, A_{\overline{\xi}} \cdot x, (\overline{\xi} \times A_{\overline{\xi}}) \cdot x)$ is the composition of a rotation by a rational orthogonal matrix mapping $\{e_1, e_2, e_3\}$ to $\{ \overline{\xi}, A_{\overline{\xi}}, \overline{\xi} \times A_{\overline{\xi}} \}$, a translation, and a rescaling by integers,  for $1 < p \leq \infty$, we have
\begin{align*}
\fint_{ \mathbb{T}^3} \eta_{(\overline{\xi})}(t,x)^2(t,x) dx = 1, \quad \|\eta_{(\overline{\xi})}\|_{L^{\infty}_t L^p_x(\mathbb{T}^3)} \lesssim r^{3/2 - 3/p}.
\end{align*}

Let $W_{(\overline{\xi})}$ be the Beltrami plane wave at frequency $\lambda$,
\begin{align*}
W_{(\overline{\xi})} = W_{\overline{\xi}, \lambda}(x) = B_{\overline{\xi}} e^{i \lambda \overline{\xi} \cdot x}.
\end{align*}
Define the intermittent Beltrami wave $\mathbb{W}_{(\overline{\xi})}$ as
\begin{align}
\mathbb{W}_{(\overline{\xi})}(t,x) := \mathbb{W}_{\overline{\xi},\lambda,\sigma,r,\mu}(t,x) = \eta_{(\overline{\xi})}(t,x)W_{(\overline{\xi})}(x).
\end{align}
It follows from the definitions and \eqref{est-xi+xi'} that
\begin{align}
\mathbb{P}_{[\frac{\lambda}{2}, 2 \lambda)} \mathbb{W}_{(\overline{\xi})} &= \mathbb{W}_{(\overline{\xi})}, \label{est-supp-frequency-W}\\
\mathbb{P}_{[\frac{\lambda}{5}, 4 \lambda)} (\mathbb{W}_{(\overline{\xi})} \otimes \mathbb{W}_{(\overline{\xi}')}) &= \mathbb{W}_{(\overline{\xi})} \otimes \mathbb{W}_{(\overline{\xi}')}, \quad \overline{\xi}' \neq -\overline{\xi} \label{est-supp-frequency-W2}.
\end{align}
The following properties are immediate from the definitions.
\begin{Proposition}(\cite[Proposition 3.4]{BV17}) \label{Prop:Property-W}
	Let $a_{\overline{\xi}} \in \mathbb{C}$ be constants with $a_{\overline{\xi}}^* = a_{-\overline{\xi}}$.  Let
	\begin{align*}
	W(x) =  \sum_{\overline{\xi} \in \Lambda} a_{\overline{\xi}} \mathbb{W}_{(\overline{\xi})}(x).
	\end{align*}
	Then $W(x)$ is real valued. Moreover, for each $R \in B_{\varepsilon_{\gamma}}(\mathrm{Id})$ we have
	\begin{align*}
	\sum_{\overline{\xi} \in \Lambda}  \left(\gamma_{(\overline{\xi})}(R)\right)^2 \fint_{ \mathbb{T}^3} \mathbb{W}_{(\overline{\xi})} \otimes \mathbb{W}_{(-\overline{\xi})} = \sum_{\overline{\xi} \in \Lambda}  \left(\gamma_{(\overline{\xi})}(R)\right)^2 B_{\overline{\xi}} \otimes B_{-\overline{\xi}} = R.
	\end{align*}
\end{Proposition}

\begin{Proposition}(\cite[Proposition 3.5]{BV17})
	For any $1 < p \leq \infty, N \geq 0, K \geq 0$:
	\begin{align}
	\| \nabla^N \ptl_t^K \mathbb{W}_{(\overline{\xi})} \|_{L^{\infty}_t L^p_x} &\lesssim \lambda^N (\lambda \sigma r \mu)^K r^{3/2 - 3/p}, \label{est-W-xi-L^p}
	\\
	\| \nabla^N \ptl_t^K \eta_{(\overline{\xi})} \|_{L^{\infty}_t L^p_x} &\lesssim (\lambda \sigma r)^N (\lambda \sigma r \mu)^K r^{3/2 - 3/p}. \label{est-eta-xi-L^p}
	\end{align}
\end{Proposition}

\subsection{Perturbations}
Let $\psi(t)$ be a smooth cut-off function such that
\begin{align}
\psi(t) = 1 \text{ on } \operatorname{supp}_t R_q,\quad  \supp \psi(t) \subset N_{\delta_{q+1}}(\operatorname{supp}_t R_q), \quad |\psi'(t)| \leq 2 \delta_{q+1}^{-1}. \label{eta_0}
\end{align}
Take a smooth increasing function $\chi$	such that
\begin{align*}
\chi(s) = \begin{cases}
1, & 0 \leq s < 1\\
s, & s \geq 2
\end{cases},
\end{align*}
and set
\begin{align*}
\rho(t,x) = \varepsilon_{\gamma}^{-1}\delta_{q+1} \chi(\delta_{q+1}^{-1}|R_q(t,x)|) \psi^2(t).
\end{align*}
where $\varepsilon_{\gamma}$ is the constant in Proposition \ref{Prop:convex-representation}.
Then clearly
\begin{align}
\operatorname{supp}_t \rho &\subset N_{\delta_{q+1}}(\operatorname{supp}_t R_q). \label{est-supp-rho}
\end{align}
It follows from the above definition that
\begin{align*}
&|R_q|/\rho = \varepsilon_{\gamma}\frac{|R_q|}{\delta_{q+1} \chi(\delta_{q+1}^{-1}|R_q(t,x)|) \psi^2} \leq \varepsilon_{\gamma} \implies
\mathrm{Id} - R_q/\rho \in B_{\varepsilon_{\gamma}}(\mathrm{Id}) \text{ on } \supp R_q.
\end{align*}
Therefore, the amplitude functions
\begin{align*}
a_{(\overline{\xi})}(t,x) := \rho^{1/2}(t,x) \gamma_{(\overline{\xi})}(\mathrm{Id} - \rho(t,x)^{-1}R_q(t,x) )
\end{align*}
are well-defined and smooth.
Define the velocity perturbation to be $w = w_{q+1}$:
\begin{align*}
w &= w^{(p)} + w^{(c)} + w^{(t)}, \\
w^{(p)} &= \sum_{\overline{\xi} \in \Lambda} a_{(\overline{\xi})} \mathbb{W}_{(\overline{\xi})} = \sum_{\overline{\xi} \in \Lambda} a_{(\overline{\xi})}(t,x) \eta_{(\overline{\xi})}(t,x) B_{\overline{\xi}} e^{i \lambda \overline{\xi} \cdot x},\\
w^{(c)} &= \frac{1}{\lambda_{q+1}} \sum_{\overline{\xi} \in \Lambda} \nabla \left( a_{(\overline{\xi})}\eta_{(\overline{\xi})} \right) \times  W_{(\overline{\xi})},\\
w^{(t)} &= \frac{1}{\mu}\sum_{\overline{\xi} \in \Lambda^+} \mathbb{P}_{LH} \mathbb{P}_{\neq 0}\left( a_{(\overline{\xi})}^2 \eta_{(\overline{\xi})}^2 \overline{\xi} \right),
\end{align*}
where  $\mathbb{P}_{LH} = \mathrm{Id} - \nabla \Delta^{-1} \operatorname{div}$ is the Leray-Helmholtz projection into divergence-free vector field, and $\mathbb{P}_{\neq 0} f = f - \fint_{ \mathbb{T}^3} f dx$. It is well-known that $\mathbb{P}_{LH}$ is bounded on $L^p, 1< p < \infty$ (see, e.g., \cite{Grafakos}).
It follows from Proposition \ref{Prop:Property-W} that
\begin{align}
\sum_{\overline{\xi} \in \Lambda}   a_{(\overline{\xi})}^2 \fint_{ \mathbb{T}^3} \mathbb{W}_{(\overline{\xi})} \otimes \mathbb{W}_{(-\overline{\xi})} dx
&= \rho \mathrm{Id} - R_q. \label{eq:rho-stress}
\end{align}

\subsection{Estimates for perturbations}

\begin{Lemma}
	The following
	bounds hold:
	\begin{align}
	\|\rho\|_{L^{\infty}_t L^1_x} &\leq C \delta_{q+1}, \label{est-rho-L1}\\
	\|\rho^{-1}\|_{C^0(\supp R_q)} & \lesssim \delta_{q+1}^{-1}, \label{est-rho-inverse}\\
	\|\rho\|_{C^N_{t,x}} & \leq C(\delta_{q+1}, \|R_q\|_{C^N}), \label{est-rho-C^N}\\
	\|a_{(\overline{\xi})}\|_{L^{\infty}_t L^2_x} &\lesssim \|\rho\|_{L^{\infty}_t L^1_x}^{1/2} \lesssim \delta_{q+1}^{1/2}, \label{est-a-L2}\\
	\|a_{(\overline{\xi})}\|_{C^N_{t,x}} & \leq C(\delta_{q+1}, \|R_q\|_{C^N}). \label{est-a-C^N}
	\end{align}
\end{Lemma}

\begin{proof}It follows from \eqref{est-R_q-L^1} that
	\begin{align*}
	\|\rho(t,\cdot)\|_{L^1_x} &=\int_{|R_q| \leq \delta_{q+1}} \rho + \int_{|R_q| > \delta_{q+1}} \rho \lesssim \delta_{q+1} + \int_{|R_q| > \delta_{q+1}} |R_q|\\
	& \leq C \delta_{q+1}.
	\end{align*}
	It is direct to verify \eqref{est-rho-inverse} and  \eqref{est-a-L2}, while \eqref{est-rho-C^N} and \eqref{est-a-C^N} follow from \eqref{eta_0} and \eqref{est-rho-inverse}.
\end{proof}

Now we can estimate the time support of $w_{q+1}$:
\begin{align}
\operatorname{supp}_t w_{q+1} \subset \operatorname{supp}_t \rho \subset \supp \psi \subset N_{\delta_{q+1}}(\operatorname{supp}_t R_q). \label{est-supp-w}
\end{align}

We need the following Lemma, which is a variant of \cite[Lemma 3.6]{BV17}.
\begin{Lemma}(\cite[Lemma 2.1]{MS17})\label{Lemma:improved-Holder}
	Let $f,g \in C^{\infty}(\mathbb{T}^3)$, and $g$ is $(\mathbb{T}/N)^3$ periodic, $N \in \mathbb{N}$. Then for $1 \leq p \leq \infty$,
	\begin{align*}
	\|f g\|_{L^p} \leq \|f\|_{L^p} \|g\|_{L^p}  + C_p N^{-1/p} \|f\|_{C^1} \|g\|_{L^p}.
	\end{align*}
\end{Lemma}
Let us denote
\begin{align}
\mathcal{C}_N =C(\sup_{\overline{\xi} \in \Lambda}\|a_{(\overline{\xi})}\|_{C^N_{t,x}})
\end{align}
to be some polynomials depending on $\sup_{\overline{\xi} \in \Lambda}\|a_{(\overline{\xi})}\|_{C^N_{t,x}}$.
\begin{Lemma}\label{Lemma:est-w}
	Suppose the parameters satisfy \eqref{ineq:parameters} and
	\begin{align}
	r^{3/2} \leq \mu .\label{ass}
	\end{align}
	Then the following estimates for the perturbations hold:
	\begin{align}
	\|w_{q+1}^{(p)}\|_{L^{\infty}_t L^2_x} & \lesssim \delta_{q+1}^{1/2} + (\lambda_{q+1} \sigma)^{-1/2} \mathcal{C}_1,\\
	\|w_{q+1}\|_{L^{\infty}_t L^p_x} & \lesssim r^{3/2-3/p}\mathcal{C}_1, \label{est-w-Lp}\\
	\|w_{q+1}^{(c)}\|_{L^{\infty}_t L^p_x} + \|w_{q+1}^{(t)}\|_{L^{\infty}_t L^p_x} & \lesssim (\sigma r + \mu^{-1}r^{3/2})r^{3/2-3/p}\mathcal{C}_1, \\
	\|\ptl_t w_{q+1}^{(p)}\|_{L^{\infty}_t L^p_x} + \|\ptl_t w_{q+1}^{(c)}\|_{L^{\infty}_t L^p_x} & \lesssim \lambda_{q+1} \sigma \mu r^{5/2 - 3/p}\mathcal{C}_2, \label{est-w1p}\\
	\| |\nabla|^N w_{q+1}\|_{L^{\infty}_t L^p_x} & \lesssim r^{3/2-3/p} \lambda_{q+1}^N \mathcal{C}_{N+1}, \label{est-nabla-s-w}
	\end{align}
	for $1 < p < \infty, N \geq 1$.
\end{Lemma}

\begin{proof}	
	Since $\mathbb{W}_{(\overline{\xi})}$ is $(\mathbb{T}/\lambda \sigma)^3$ periodic, it follows from \eqref{est-W-xi-L^p},  \eqref{est-a-L2}, and Lemma \ref{Lemma:improved-Holder} that
	\begin{align*}
	\|w_{q+1}^{(p)}\|_{L^{\infty}_t L^2_x} & \lesssim \sum_{\overline{\xi} \in \Lambda} (\|a_{(\overline{\xi})}\|_{L^{\infty}_t L^2_x}    + (\lambda_{q+1} \sigma)^{-1/2}  \|a_{(\overline{\xi})}\|_{C^1}) \|\mathbb{W}_{(\overline{\xi})}\|_{L^{\infty}_t L^2_x}\\
	& \lesssim \delta_{q+1}^{1/2} + (\lambda_{q+1} \sigma)^{-1/2}\mathcal{C}_1.
	\end{align*}
	In view of \eqref{ineq:parameters}, \eqref{est-W-xi-L^p} and \eqref{est-eta-xi-L^p} yield that
	\begin{align*}
	\|w_{q+1}^{(p)}\|_{L^{\infty}_t L^p_x} & \lesssim \sum_{\overline{\xi} \in \Lambda} \|a_{(\overline{\xi})}\|_{C^0}\|\mathbb{W}_{(\overline{\xi})}\|_{L^{\infty}_t L^p_x}\lesssim r^{3/2-3/p}\mathcal{C}_0,\\
	\|w_{q+1}^{(c)}\|_{L^{\infty}_t L^p_x}  & \lesssim \lambda_{q+1}^{-1} \sum_{\overline{\xi} \in \Lambda} \left(\| \eta_{(\overline{\xi})} \|_{L^{\infty}_t L^p_x} + \| \nabla \eta_{(\overline{\xi})} \|_{L^{\infty}_t L^p_x}\right)\|a_{(\overline{\xi})}\|_{C^1}\|\mathbb{W}_{(\overline{\xi})}\|_{L^{\infty}_t L^p_x}\\
	& \lesssim (\sigma r) r^{3/2-3/p}\mathcal{C}_1,\\
	\|w_{q+1}^{(t)}\|_{L^{\infty}_t L^p_x} &  \lesssim \mu^{-1}\sum_{\overline{\xi} \in \Lambda^+} \|a_{(\overline{\xi})}^2 \eta_{(\overline{\xi})}^2 \overline{\xi} \|_{L^{\infty}_t L^p_x} \lesssim  \mu^{-1} \sum_{\overline{\xi} \in \Lambda^+}\|a_{(\overline{\xi})}^2\|_{C^0} \|  \eta_{(\overline{\xi})} \|_{L^{\infty}_t L^{2p}_x}^2 \\
	&\lesssim \mu^{-1}r^{3-3/p}\mathcal{C}_0,
	\end{align*}
	where the boundedness of $\mathbb{P}_{LH}$ and $\mathbb{P}_{\neq 0}$ on $L^p$, for $1 < p < \infty$,  is used in the first inequality of the estimate for $\|w_{q+1}^{(t)}\|_{L^{\infty}_t L^p_x}$. In the same way, we can estimate
	\begin{align*}
	\|\ptl_t w_{q+1}^{(p)}\|_{L^{\infty}_t L^p_x}  & \lesssim \sum_{\overline{\xi} \in \Lambda} \|\ptl_t a_{(\overline{\xi})}\|_{C^0} \| \mathbb{W}_{(\overline{\xi})}\|_{L^{\infty}_t L^p_x} + \|a_{(\overline{\xi})}\|_{C^0} \|\ptl_t \mathbb{W}_{(\overline{\xi})}\|_{L^{\infty}_t L^p_x}\\
	& \lesssim  \lambda_{q+1} \sigma \mu r^{5/2 - 3/p} \mathcal{C}_1,\\
	\|\ptl_t w_{q+1}^{(c)}\|_{L^{\infty}_t L^p_x}  & \lesssim \lambda_{q+1}^{-1} \sum_{\overline{\xi} \in \Lambda} \| a_{(\overline{\xi})}\|_{C^2_{t,x}}  \Big( \| \eta_{(\overline{\xi})} \|_{L^{\infty}_t L^p_x} + \| \nabla \eta_{(\overline{\xi})} \|_{L^{\infty}_t L^p_x} + \|\ptl_t \eta_{(\overline{\xi})} \|_{L^{\infty}_t L^p_x}    \\
	& \qquad + \| \ptl_t \nabla \eta_{(\overline{\xi})}\|_{L^{\infty}_t L^p_x}\Big) \lesssim \sigma r\lambda_{q+1} \sigma \mu r^{5/2 - 3/p} \mathcal{C}_2 \lesssim \lambda_{q+1} \sigma \mu r^{5/2 - 3/p}\mathcal{C}_2.
	\end{align*}
	For $N \geq 1$, Using \eqref{est-W-xi-L^p} and \eqref{est-eta-xi-L^p}, we obtain that
	\begin{align*}
	\|\nabla^N w_{q+1}^{(p)}\|_{L^{\infty}_t L^p_x}  & \lesssim \sum_{\overline{\xi} \in \Lambda} \sum_{k=0}^N \|\nabla^k a_{(\overline{\xi})}\|_{C^0} \| \nabla^{N-k} \mathbb{W}_{(\overline{\xi})}\|_{L^{\infty}_t L^p_x} \\
	& \lesssim  \lambda_{q+1}^N  r^{3/2 - 3/p}\mathcal{C}_N,\\
	\|\nabla^N w_{q+1}^{(c)}\|_{L^{\infty}_t L^p_x}  & \lesssim \lambda_{q+1}^{-1} \sum_{\overline{\xi} \in \Lambda} \sum_{m=0}^{N}\sum_{k=0}^m \lambda_{q+1}^{N-m}  \|\nabla^{k+1} a_{(\overline{\xi})}\|_{C^0} \| \nabla^{m-k} \eta_{(\overline{\xi})}\|_{L^{\infty}_t L^p_x} \\
	& \quad + \lambda_{q+1}^{-1} \sum_{\overline{\xi} \in \Lambda} \sum_{m=0}^{N}\sum_{k=0}^m \lambda_{q+1}^{N-m}  \|\nabla^{k} a_{(\overline{\xi})}\|_{C^0} \| \nabla^{m-k+1} \eta_{(\overline{\xi})}\|_{L^{\infty}_t L^p_x}\\
	& \lesssim \lambda_{q+1}^N  r^{3/2 - 3/p} \mathcal{C}_{N+1},\\
	\|\nabla^N w_{q+1}^{(t)}\|_{L^{\infty}_t L^p_x}  & \lesssim  \mu^{-1}\sum_{\overline{\xi} \in \Lambda} \sum_{m=0}^{N} \| \nabla^{N-m}(a_{(\overline{\xi})}^2)\|_{C^0}  \sum_{k=0}^m \|\nabla^{k}\eta_{(\overline{\xi})}\|_{L^{\infty}_t L^{2p}_x} \|\nabla^{m-k}\eta_{(\overline{\xi})}\|_{L^{\infty}_t L^{2p}_x}\\
	&\lesssim \lambda_{q+1}^N  r^{3/2 - 3/p}\frac{(\sigma r)^N r^{3/2}}{\mu} \mathcal{C}_N  \lesssim \lambda_{q+1}^N  r^{3/2 - 3/p}\mathcal{C}_N,
	\end{align*}
	where we use \eqref{ineq:parameters} and \eqref{ass}.
 \end{proof}

\subsection{Estimates for the stress}
Let us recall the following operator in \cite{dLSz4}.
\begin{Lemma}[symmetric anti-divergence]\label{Lemma:symm-anti-div} There exists a linear operator $\mathcal{R}$, of order $-1$, mapping vector fields to symmetric matrices such that
	\begin{align}
	\nabla \cdot \mathcal{R}(u) = u - \fint_{ \mathbb{T}^3} u, \label{eq:mcR}
	\end{align}
	with standard Calderon-Zygmund estimates, for $1 < p < \infty$,
	\begin{equation}
	\| \mathcal{R} \|_{L^p \to W^{1,p}} \lesssim 1, \quad \|  \mathcal{R} \|_{C^0 \to C^0} \lesssim 1,\quad
	\| \mathcal{R} \mathbb{P}_{\neq 0} u \|_{L^p} \lesssim \| |\nabla|^{-1}\mathbb{P}_{\neq 0} u\|_{L^p}. \label{est-mcR}
	\end{equation}	
\end{Lemma}

\begin{proof}
	Suppose $u \in C^{\infty}(\mathbb{T}^3, \mathbb{R}^3)$ is a smooth vector field. Define
	\begin{align*}
	 \mathcal{R}(u) = \frac{1}{4}\left(\nabla \mathbb{P}_{LH} v + (\nabla \mathbb{P}_{LH} v)^T\right) + \frac{3}{4}\left(\nabla v + (\nabla v)^T\right) - \frac{1}{2}(\nabla \cdot v) \mathrm{Id}
	\end{align*}
	where $v \in C^{\infty}(\mathbb{T}^3, \mathbb{R}^3)$ is the unique solution to $\Delta v = u - \fint_{ \mathbb{T}^3} u$ with $\fint_{ \mathbb{T}^3} v = 0$.
	
	It is direct to verify that $\mathcal{R}(u)$ is a symmetric matrix field depending linearly on $u$ and satisfies \eqref{eq:mcR}. Note that $\mathcal{R}$ is a constant coefficient ellitpic operator of order $-1$.  We refer to \cite{Grafakos} for the Calderon-Zygmund estimates $\| \mathcal{R} \|_{L^p \to W^{1,p}} \lesssim 1$ and $\| \mathcal{R} \mathbb{P}_{\neq 0} u \|_{L^p} \lesssim \| |\nabla|^{-1}\mathbb{P}_{\neq 0} u\|_{L^p}$. Combining these with Sobolev embeddings, we have $\| \mathcal{R} u\|_{C^{\alpha}} \lesssim \|\mathcal{R} u\|_{W^{1,4}} \lesssim \| u\|_{L^4} \lesssim \| u\|_{C^0}$, with $\alpha = 1/4$.
\end{proof}

We have the following variant of \cite[Lemma B.1]{BV17} in \cite{BV17}.
\begin{Lemma}\label{Lemma:commutator-est}
	Let $a\in C^2(\mathbb{T}^3)$.
	For $1 < p < \infty$, and any smooth function $f \in L^p(\mathbb{T}^3)$, we have
	\begin{align}
	\| |\nabla|^{-1}\mathbb{P}_{\neq 0}(a \mathbb{P}_{\geq k} f)\|_{L^p(\mathbb{T}^3)} \lesssim k^{-1} \|\nabla^2 a\|_{L^{\infty}(\mathbb{T}^3)} \|f\|_{L^p(\mathbb{T}^3)}. \label{est-commutator-1}
	\end{align}
\end{Lemma}

\begin{proof}[Proof of Lemma \ref{Lemma:commutator-est}]
	We follow the proof in \cite{BV17}.
	Note that
	\begin{align*}
	|\nabla|^{-1}\mathbb{P}_{\neq 0}(a \mathbb{P}_{\geq k} f) = |\nabla|^{-1}\mathbb{P}_{\geq k/2}(\mathbb{P}_{\leq k/2} a \mathbb{P}_{\geq k} f) + |\nabla|^{-1}\mathbb{P}_{\neq 0}(\mathbb{P}_{\geq k/2} a \mathbb{P}_{\geq k} f).
	\end{align*}
	As direct consequences of the Littlewood-Paley decomposition and Schauder estimates we have the bounds for $1 < p < \infty$ (see, for example, \cite{Grafakos})
	\begin{align*}
	\|\mathbb{P}_{\leq k/2}\|_{L^p \to L^p} \lesssim 1, \quad \| |\nabla|^{-1}\mathbb{P}_{\geq k/2} \|_{L^p \to L^p} \lesssim k^{-1}, \quad \| |\nabla|^{-1}\mathbb{P}_{\neq 0} \|_{L^p \to L^p} \lesssim 1.
	\end{align*}
	Combining these bounds with H\"{o}lder's inequality and the embedding $W^{1,4}(\mathbb{T}^3) \subset L^{\infty}(\mathbb{T}^3)$, we obtain
	\begin{align*}
	& \quad \| |\nabla|^{-1}\mathbb{P}_{\neq 0}(a \mathbb{P}_{\geq k} f)\|_{L^p}  \lesssim k^{-1} \| \mathbb{P}_{\leq k/2} a \mathbb{P}_{\geq k} f \|_{L^p} + \| \mathbb{P}_{\geq k/2} a \mathbb{P}_{\geq k} f \|_{L^p}  \\
	& \lesssim k^{-1}  (\|\mathbb{P}_{\leq k/2} a\|_{L^{\infty}}  + k\| \mathbb{P}_{\geq k/2} a\|_{L^{\infty}}) \|f\|_{L^p} \lesssim k^{-1}  (\|\nabla \mathbb{P}_{\leq k/2} a\|_{L^4}  + k\|\nabla \mathbb{P}_{\geq k/2} a\|_{L^4}) \|f\|_{L^p}\\
	& \lesssim k^{-1}  (\|\mathbb{P}_{\leq k/2} \nabla  a\|_{L^4}  + k\||\nabla|^{-1} \mathbb{P}_{\geq k/2} |\nabla|\nabla  \mathbb{P}_{\geq k/2}  a\|_{L^4}) \|f\|_{L^p}  \\
	&\lesssim k^{-1}  (\| \nabla  a\|_{L^4}  + \|\nabla^2 \mathbb{P}_{\geq k/2} a\|_{L^4}) \|f\|_{L^p} \lesssim k^{-1}  \| \nabla^2  a\|_{L^4} \|f\|_{L^p}.
	\end{align*}
\end{proof}

It follows from the definition of $w_{q+1}$ that
\begin{align*}
\int_{\mathbb{T}^3}  w_{q+1} dx = \int_{\mathbb{T}^3} \frac{1}{\lambda_{q+1}} \sum_{\overline{\xi} \in \Lambda} \nabla \left( a_{(\overline{\xi})}\eta_{(\overline{\xi})}  W_{(\overline{\xi})} \right) dx   + \int_{\mathbb{T}^3} \frac{1}{\mu}\sum_{\overline{\xi} \in \Lambda^+}{P}_{LH} \mathbb{P}_{\neq 0}\left( a_{(\overline{\xi})}^2 \eta_{(\overline{\xi})}^2 \overline{\xi} \right)dx = 0.
\end{align*}
Hence $\int_{\mathbb{T}^3} \nu (-\Delta)^{\theta} w_{q+1}dx = 0$ and $\dfrac{d}{dt} \int_{\mathbb{T}^3}  w_{q+1} dx = 0$.
We obtain $R_{q+1}$ by plugging $v_{q+1} = v_q + w_{q+1}$ in \eqref{eq:NS-alpha-Reynold-stress}, using \eqref{eq:mcR} and the assumption that $(v_q,R_q)$ solves \eqref{eq:NS-alpha-Reynold-stress}:
\begin{align*}
\nabla \cdot R_{q+1} &= \nabla \cdot \left[\mathcal{R}( \nu(-\Delta)^{\theta} w_{q+1} + \ptl_t w_{q+1}^{(p)} + \ptl_t w_{q+1}^{(c)}) +  v_q \otimes w_{q+1} + w_{q+1} \otimes v_q \right]\\
& \quad + \nabla \cdot \left[(w_{q+1}^{(c)} + w_{q+1}^{(t)}) \otimes w_{q+1} + w_{q+1}^{(p)} \otimes (w_{q+1}^{(c)} + w_{q+1}^{(t)})\right]\\
& \quad \left[\nabla \cdot (w_{q+1}^{(p)} \otimes w_{q+1}^{(p)} - R_{q}) + \ptl_t w_{q+1}^{(t)}\right]  + \nabla (p_{q+1} - p_q)\\
& := \nabla \cdot (\widetilde{R}_{linear} + \widetilde{R}_{corrector} + \widetilde{R}_{oscillation})   + \nabla(p_{q+1}-p_q).
\end{align*}

It follows from Lemma \ref{Lemma:est-w} that
\begin{align*}
\| \widetilde{R}_{corrector} \|_{L^{\infty}_t L^p_x} &\lesssim \left(\|w_{q+1}^{(c)}\|_{L^{\infty}_t L^{2p}_x} + \|w_{q+1}^{(t)}\|_{L^{\infty}_t L^{2p}_x} \right) \left(\|w_{q+1}\|_{L^{\infty}_t L^{2p}_x} + \|w_{q+1}^{(p)}\|_{L^{\infty}_t L^{2p}_x}\right) \\
&\lesssim (\sigma r + \mu^{-1}r^{3/2})r^{3-3/p}\mathcal{C}_1.
\end{align*}

Noting that $\nabla \times \dfrac{w_{q+1}^{(p)}}{\lambda_{q+1}} = w_{q+1}^{(p)} + w_{q+1}^{(c)}$, Lemma \ref{Lemma:est-w} and \eqref{est-mcR} yield that
\begin{align}
&\|\widetilde{R}_{linear}\|_{L^{\infty}_t L^p_x} \nonumber\\
&\lesssim \lambda_{q+1}^{-1}\|\ptl_t \mathcal{R} \nabla \times  (w_{q+1}^{(p)})\|_{L^{\infty}_t L^p_x} + \|\mathcal{R}(\nu(-\Delta)^{\theta}w_{q+1})\|_{L^{\infty}_t L^p_x} \nonumber\\
&\quad + \|v_q \otimes w_{q+1} + w_{q+1} \otimes v_q\|_{L^{\infty}_t L^p_x} \nonumber\\
&\lesssim \lambda_{q+1}^{-1}\|\ptl_t w_{q+1}^{(p)}\|_{L^{\infty}_t L^p_x} + \| |\nabla|^{2\theta - 1} w_{q+1}\|_{L^{\infty}_t L^p_x} +  \|v_q\|_{C^0} \|w_{q+1}\|_{L^{\infty}_t L^p_x}\nonumber \\
& \lesssim
\sigma \mu r^{5/2-3/p} \mathcal{C}_2
+ r^{3/2-3/p}(\lambda_{q+1}^{2\theta-1} + \|v_q\|_{C^0})\mathcal{C}_3. \label{est-R-linear}
\end{align}
This is the crucial estimate to control the fractional viscosity. If we assume that $p \sim 1, r \sim \lambda_{q+1}^{-1}$, we must have $\theta < 5/4$ in order that the second term in \eqref{est-R-linear} is small for $\lambda_{q+1}$ sufficiently large.  
%

It remains to estimate 	$\widetilde{R}_{oscillation}$, which can be handled in the same way as in \cite{BV17}.
It follows from \eqref{eq:rho-stress} that
\begin{align*}
 & \nabla \cdot (w_{q+1}^{(p)} \otimes w_{q+1}^{(p)} - R_{q}) = \nabla \cdot (\sum_{\overline{\xi}, \overline{\xi}' \in \Lambda} a_{(\overline{\xi})} a_{(\overline{\xi}')} \mathbb{W}_{\overline{\xi}} \otimes \mathbb{W}_{(\overline{\xi}')} - R_q)\\
&=   \nabla \cdot (\sum_{\overline{\xi}, \overline{\xi}' \in \Lambda} a_{(\overline{\xi})} a_{(\overline{\xi}')} \mathbb{P}_{\geq \lambda_{q+1} \sigma/2}\mathbb{W}_{(\overline{\xi})} \otimes \mathbb{W}_{(\overline{\xi}')}  ) + \nabla \rho\\
&:= \sum_{\overline{\xi}, \overline{\xi}' \in \Lambda} E_{(\overline{\xi}, \overline{\xi}')}  + \nabla \rho.
\end{align*}
Since $E_{(\overline{\xi}, \overline{\xi}')}$ has zero mean, we can split it as
\begin{align*}
E_{(\overline{\xi}, \overline{\xi}')} + E_{(\overline{\xi}', \overline{\xi})} &= \mathbb{P}_{\neq 0} \left( \nabla  (a_{(\overline{\xi})} a_{(\overline{\xi}')}) \cdot (\mathbb{P}_{\geq \lambda_{q+1} \sigma/2}(\mathbb{W}_{(\overline{\xi})} \otimes \mathbb{W}_{(\overline{\xi}')} + \mathbb{W}_{\overline{\xi'}} \otimes \mathbb{W}_{(\overline{\xi})}) ) \right)\\
& \quad + \mathbb{P}_{\neq 0} \left( a_{(\overline{\xi})} a_{(\overline{\xi}')} \nabla \cdot (\mathbb{W}_{(\overline{\xi})} \otimes \mathbb{W}_{(\overline{\xi}')} + \mathbb{W}_{\overline{\xi'}} \otimes \mathbb{W}_{(\overline{\xi})})  \right) \\
&:= E_{(\overline{\xi}, \overline{\xi}',1)} + E_{(\overline{\xi}, \overline{\xi}',2)}.
\end{align*}

Using \eqref{est-W-xi-L^p}, \eqref{est-mcR} and \eqref{est-commutator-1}, we obtain
\begin{align*}
\|\mathcal{R} E_{(\overline{\xi}, \overline{\xi}', 1)}\|_{L^{\infty}_t L^p_x} & \lesssim \| |\nabla|^{-1} E_{(\overline{\xi}, \overline{\xi}', 1)}\|_{L^{\infty}_t L^p_x} \\
& \lesssim  (\lambda_{q+1} \sigma)^{-1} \|a_{(\overline{\xi})} a_{(\overline{\xi}')}\|_{C^3} \|\mathbb{W}_{(\overline{\xi})} \otimes \mathbb{W}_{(\overline{\xi}')}\|_{L^{\infty}_t L^p_x}\\
& \lesssim  (\lambda_{q+1} \sigma)^{-1} \|a_{(\overline{\xi})} a_{(\overline{\xi}')}\|_{C^3} \|\mathbb{W}_{(\overline{\xi})}\|_{L^{\infty}_t L^{2p}_x} \| \mathbb{W}_{(\overline{\xi}')}\|_{L^{\infty}_t L^{2p}_x}\\
& \lesssim  (\lambda_{q+1} \sigma)^{-1}  r^{3-3/p} \mathcal{C}_3.
\end{align*}

Recall the vector identity $A \cdot \nabla B + B \cdot \nabla A = \nabla (A \cdot B) - A \times (\nabla \times B) - B \times (\nabla \times A)$.
For $\overline{\xi}, \overline{\xi}' \in \Lambda$, using the anti-symmetry of the cross product,  we can write
\begin{align*}
& \nabla \cdot (\mathbb{W}_{(\overline{\xi})} \otimes \mathbb{W}_{(\overline{\xi}')} + \mathbb{W}_{(\overline{\xi}')} \otimes \mathbb{W}_{(\overline{\xi})} )\\
&= \left( W_{(\overline{\xi})} \otimes W_{(\overline{\xi}')} + W_{(\overline{\xi}')} \otimes W_{(\overline{\xi})} \right) \nabla \left( \eta_{(\overline{\xi})} \eta_{(\overline{\xi}')} \right) +  \eta_{(\overline{\xi})} \eta_{(\overline{\xi}')}  \left( W_{(\overline{\xi})}  \cdot \nabla W_{(\overline{\xi}')} + W_{(\overline{\xi}')} \cdot \nabla W_{(\overline{\xi})} \right)\\
&= \left( W_{(\overline{\xi}')} \cdot \nabla \left( \eta_{(\overline{\xi})} \eta_{(\overline{\xi}')} \right) \right) W_{(\overline{\xi})} + \left( W_{(\overline{\xi})} \cdot \nabla \left( \eta_{(\overline{\xi})} \eta_{(\overline{\xi}')} \right) \right) W_{\overline{\xi}' } + \eta_{(\overline{\xi})} \eta_{(\overline{\xi}')} \nabla \left( W_{(\overline{\xi})} \cdot W_{(\overline{\xi}')} \right).
\end{align*}

For the term $E_{(\overline{\xi}, \overline{\xi}',2)}$, first consider the case $\overline{\xi} + \overline{\xi'} \neq 0$. It follows from the above identity and \eqref{est-supp-frequency-W2} that
\begin{align*}
& \quad a_{(\overline{\xi})} a_{(\overline{\xi}')} \nabla \cdot (\mathbb{W}_{(\overline{\xi})} \otimes \mathbb{W}_{(\overline{\xi}')} + \mathbb{W}_{(\overline{\xi}')} \otimes \mathbb{W}_{(\overline{\xi})})\\
&= a_{(\overline{\xi})} a_{(\overline{\xi}')} \nabla \cdot \mathbb{P}_{\geq \lambda_{q+1}/10} \left(\eta_{(\overline{\xi})} \eta_{(\overline{\xi}')} \left( W_{(\overline{\xi})} \otimes W_{(\overline{\xi}')} + W_{(\overline{\xi}')} \otimes W_{(\overline{\xi})}  \right) \right)\\
& = a_{(\overline{\xi})} a_{(\overline{\xi}')} \mathbb{P}_{\geq \lambda_{q+1}/10} \left( \nabla \left( \eta_{(\overline{\xi})} \eta_{(\overline{\xi}')} \right) \cdot \left( W_{(\overline{\xi})} \otimes W_{(\overline{\xi}')} + W_{(\overline{\xi}')} \otimes W_{(\overline{\xi})}  \right) \right)\\
& \quad + a_{(\overline{\xi})} a_{(\overline{\xi}')} \mathbb{P}_{\geq \lambda_{q+1}/10} \left( \eta_{(\overline{\xi})} \eta_{(\overline{\xi}')} \nabla  \left(W_{(\overline{\xi})} \cdot W_{(\overline{\xi}')}\right) \right) \\
& = a_{(\overline{\xi})} a_{(\overline{\xi}')} \mathbb{P}_{\geq \lambda_{q+1}/10} \left( \nabla \left( \eta_{(\overline{\xi})} \eta_{(\overline{\xi}')} \right) \cdot \left( W_{(\overline{\xi})} \otimes W_{(\overline{\xi}')} + W_{(\overline{\xi}')} \otimes W_{(\overline{\xi})}  \right) \right)\\
& \quad + \nabla \left( a_{(\overline{\xi})} a_{(\overline{\xi}')}  \mathbb{W}_{(\overline{\xi})} \cdot \mathbb{W}_{(\overline{\xi}')}  \right) - \nabla \left(a_{(\overline{\xi})} a_{(\overline{\xi}')}\right)  \mathbb{P}_{\geq \lambda_{q+1}/10}\left(\mathbb{W}_{(\overline{\xi})} \cdot \mathbb{W}_{(\overline{\xi}')} \right)\\
& \quad - a_{(\overline{\xi})} a_{(\overline{\xi}')} \mathbb{P}_{\geq \lambda_{q+1}/10} \left( \left(W_{(\overline{\xi})} \cdot W_{(\overline{\xi}')}\right) \nabla\left(\eta_{(\overline{\xi})} \eta_{(\overline{\xi}')}\right) \right),
\end{align*}
where the second term is a pressure, the third can be estimated analogously to $E_{(\overline{\xi}, \overline{\xi}', 1)}$. Also note that the first and fourth term can estimated analogously. Using \eqref{est-eta-xi-L^p}, \eqref{est-mcR} and \eqref{est-commutator-1}, we obtain
\begin{align*}
& \quad \|\mathcal{R} \left( a_{(\overline{\xi})} a_{(\overline{\xi}')} \mathbb{P}_{\geq \lambda_{q+1}/10} \left( \nabla \left( \eta_{(\overline{\xi})} \eta_{(\overline{\xi}')} \right) \cdot \left( W_{(\overline{\xi})} \otimes W_{(\overline{\xi}')} + W_{(\overline{\xi}')} \otimes W_{(\overline{\xi})}  \right) \right) \right)\|_{L^{\infty}_t L^p_x}\\
& \lesssim \lambda_{q+1}^{-1} \|a_{(\overline{\xi})} a_{(\overline{\xi}')}\|_{C^3}  \|\nabla \left( \eta_{(\overline{\xi})} \eta_{(\overline{\xi}')} \right)\|_{L^{\infty}_t L^p_x}\\
& \lesssim   \sigma r^{4-3/p} \mathcal{C}_3.
\end{align*}

Now consider $E_{(\overline{\xi}, -\overline{\xi},2)}$.
We can write
\begin{align*}
&\nabla \cdot (\mathbb{W}_{(\overline{\xi})} \otimes \mathbb{W}_{(-\overline{\xi})} + \mathbb{W}_{(-\overline{\xi})} \otimes \mathbb{W}_{(\overline{\xi})} ) = \left( W_{(-\overline{\xi})} \cdot \nabla \eta_{(\overline{\xi})}^2  \right) W_{(\overline{\xi})} + \left( W_{(\overline{\xi})} \cdot \nabla \eta_{(\overline{\xi})}^2  \right) W_{(-\overline{\xi})}\\
&= (A_{\overline{\xi}} \cdot \nabla \eta_{(\overline{\xi})}^2) A_{\overline{\xi}} + ((\overline{\xi} \times A_{\overline{\xi}}) \cdot \nabla \eta_{(\overline{\xi})}^2) (\overline{\xi} \times A_{\overline{\xi}}) = \nabla \xi_{(\overline{\xi})}^2 - (\overline{\xi}  \cdot \nabla \eta_{(\overline{\xi})}^2) \overline{\xi} = \nabla \eta_{(\overline{\xi})}^2 - \frac{\overline{\xi}}{\mu}\ptl_t \eta_{(\overline{\xi})}^2,
\end{align*}
where we use \eqref{eq:Dt-eta} and the fact that $\{ \overline{\xi}, A_{\overline{\xi}}, \overline{\xi} \times A_{\overline{\xi}} \}$ forms an orthonormal basis of $\mathbb{R}^3$. Therefore, we can write
\begin{align*}
E_{(\overline{\xi}, -\overline{\xi},2)} &= \mathbb{P}_{\neq 0}\left( a_{(\overline{\xi})}^2 \nabla \mathbb{P}_{\geq \lambda_{q+1} \sigma/2} \eta_{(\overline{\xi})}^2 - a_{(\overline{\xi})}^2 \frac{\overline{\xi}}{\mu}\ptl_t \eta_{(\overline{\xi})}^2  \right)\\
&= \nabla \left( a_{(\overline{\xi})}^2 \mathbb{P}_{\geq \lambda_{q+1} \sigma/2} \eta_{(\overline{\xi})}^2 \right) - \mathbb{P}_{\neq 0}\left(\mathbb{P}_{\geq \lambda_{q+1} \sigma/2}(\eta_{(\overline{\xi})}^2) \nabla a_{(\overline{\xi})}^2 \right)\\
& \quad - \mu^{-1} \ptl_t \mathbb{P}_{\neq 0}\left(a_{(\overline{\xi})}^2 \eta_{(\overline{\xi})}^2 \overline{\xi} \right) + \mu^{-1} \mathbb{P}_{\neq 0} \left( \ptl_t\left(a_{(\overline{\xi})}^2\right) \eta_{(\overline{\xi})}^2 \overline{\xi} \right).
\end{align*}
Using the identity $\mathrm{Id} - \mathbb{P}_{LH} = \nabla \Delta^{-1} \operatorname{div}$ , we obtain
\begin{align*}
&\sum_{\overline{\xi}} E_{(\overline{\xi}, -\overline{\xi},2)} + \ptl_t w_{q+1}^{(t)} = \nabla \sum_{\overline{\xi}} \left( a_{(\overline{\xi})}^2 \mathbb{P}_{\geq \lambda_{q+1} \sigma/2} \eta_{(\overline{\xi})}^2 \right) - \nabla \sum_{\overline{\xi}} \mu^{-1} \Delta^{-1} \nabla \cdot \ptl_t \left(a_{(\overline{\xi})}^2 \eta_{(\overline{\xi})}^2 \overline{\xi}\right)\\
& \quad - \sum_{\overline{\xi}} \mathbb{P}_{\neq 0}\left(\mathbb{P}_{\geq \lambda_{q+1} \sigma/2}(\eta_{(\overline{\xi})}^2) \nabla a_{(\overline{\xi})}^2 \right) + \mu^{-1} \sum_{\overline{\xi}}  \mathbb{P}_{\neq 0} \left( \ptl_t\left(a_{(\overline{\xi})}^2\right) \eta_{(\overline{\xi})}^2 \overline{\xi} \right),
\end{align*}
where the first and second terms are pressure terms. Using \eqref{est-eta-xi-L^p}, \eqref{est-mcR} and \eqref{est-commutator-1}, we obtain
\begin{align*}
\|\mathcal{R} \mathbb{P}_{\neq 0}\left(\mathbb{P}_{\geq \lambda_{q+1} \sigma/2}(\eta_{(\overline{\xi})}^2) \nabla a_{(\overline{\xi})}^2 \right)\|_{L^{\infty}_t L^p_x} &\lesssim (\lambda_{q+1} \sigma)^{-1} \| \eta_{(\overline{\xi})} \|_{L^{\infty}_t L^{2p}_x}^2\mathcal{C}_3  \\
&\lesssim (\lambda_{q+1} \sigma)^{-1} r^{3-3/p} \mathcal{C}_3.
\end{align*}
It follows from \eqref{est-eta-xi-L^p} and \eqref{est-mcR}  that
\begin{align*}
\mu^{-1}\|\mathcal{R} \mathbb{P}_{\neq 0} \left( \ptl_t\left(a_{(\overline{\xi})}^2\right) \eta_{(\overline{\xi})}^2 \overline{\xi} \right) \|_{L^{\infty}_t L^p_x} &\lesssim \mu^{-1}\| \ptl_t\left(a_{(\overline{\xi})}^2\right) \eta_{(\overline{\xi})}^2 \overline{\xi} \|_{L^{\infty}_t L^p_x}  \\
&\lesssim \mu^{-1} r^{3-3/p} \mathcal{C}_1.
\end{align*}

Let us now give the explicit definition of $\widetilde{R}_{oscillation}$:
\begin{align*}
&\widetilde{R}_{oscillation} =  \sum_{\overline{\xi}, \overline{\xi}' \in \Lambda} \mathbb{P}_{\neq 0} \left( \nabla  (a_{(\overline{\xi})} a_{(\overline{\xi}')}) \cdot (\mathbb{P}_{\geq \lambda_{q+1} \sigma/2}(\mathbb{W}_{(\overline{\xi})} \otimes \mathbb{W}_{(\overline{\xi}')} + \mathbb{W}_{\overline{\xi'}} \otimes \mathbb{W}_{(\overline{\xi})}) ) \right)\\
&+ \sum_{\overline{\xi}, \overline{\xi}' \in \Lambda, \overline{\xi} \neq \overline{\xi}' } a_{(\overline{\xi})} a_{(\overline{\xi}')} \mathbb{P}_{\geq \lambda_{q+1}/10} \left( \nabla \left( \eta_{(\overline{\xi})} \eta_{(\overline{\xi}')} \right) \cdot \left( W_{(\overline{\xi})} \otimes W_{(\overline{\xi}')} + W_{(\overline{\xi}')} \otimes W_{(\overline{\xi})}  \right) \right)\\
&  - \sum_{\overline{\xi}, \overline{\xi}' \in \Lambda, \overline{\xi} \neq \overline{\xi}' } \nabla \left(a_{(\overline{\xi})} a_{(\overline{\xi}')}\right)  \mathbb{P}_{\geq \lambda_{q+1}/10}\left(\mathbb{W}_{(\overline{\xi})} \cdot \mathbb{W}_{(\overline{\xi}')} \right)\\
&  - \sum_{\overline{\xi}, \overline{\xi}' \in \Lambda, \overline{\xi} \neq \overline{\xi}' } a_{(\overline{\xi})} a_{(\overline{\xi}')} \mathbb{P}_{\geq \lambda_{q+1}/10} \left( \left(W_{(\overline{\xi})} \cdot W_{(\overline{\xi}')}\right) \nabla\left(\eta_{(\overline{\xi})} \eta_{(\overline{\xi}')}\right) \right)\\
& - \sum_{\overline{\xi} \in \Lambda} \mathbb{P}_{\neq 0}\left(\mathbb{P}_{\geq \lambda_{q+1} \sigma/2}(\eta_{(\overline{\xi})}^2) \nabla a_{(\overline{\xi})}^2 \right) + \mu^{-1} \sum_{\overline{\xi} \in \Lambda}  \mathbb{P}_{\neq 0} \left( \ptl_t\left(a_{(\overline{\xi})}^2\right) \eta_{(\overline{\xi})}^2 \overline{\xi} \right).
\end{align*}

Finally, we estimate the time support of $R_{q+1}$. Using \eqref{est-supp-w} we obtain
\begin{align*}
\operatorname{supp}_t R_{q+1} \subset \operatorname{supp}_t w_{q+1} \cup \operatorname{supp}_t R_{q}  \subset N_{\delta_{q+1}}(\operatorname{supp}_t R_q).
\end{align*}

Now we choose the parameters $r, \sigma, \mu$. Fix $\alpha$ so that
\begin{align*}
\max\{0,\frac{2}{3}(2\theta - 1)\} < \alpha < 1,
\end{align*}
which is possible since $\theta \in (-\infty,5/4)$. Fix
\begin{align}
r = \lambda^{\alpha}_{q+1}, \quad \sigma = \lambda^{-(\alpha + 1)/2}_{q+1}, \quad \mu =  \lambda^{(5\alpha + 1)/4}_{q+1}.
\end{align}
Clearly \eqref{ass} is satisfied.
Choose $p > 1$ sufficiently close to $1$ so that
\begin{align*}
-\frac{\alpha+1}{2} + \frac{5\alpha + 1}{4} +\left(\frac{5}{2} - \frac{3}{p}\right)\alpha < 0
, \quad \left(\frac{3}{2} - \frac{3}{p}\right)\alpha + \max(0,2\theta - 1) < 0, \\
-\frac{5\alpha + 1}{4} + \left(\frac{9}{2} - \frac{3}{p}\right)\alpha < 0,  \quad -\frac{1-\alpha}{2} + \left(3 - \frac{3}{p}\right)\alpha < 0.
\end{align*}
Note that $\mathcal{C}_N$ is independent of $\lambda_{q+1}$, due to \eqref{est-a-C^N}. Combining the above estimates with Lemma \ref{Lemma:est-w}, it is easy to check that, by taking $\lambda_{q+1}$ sufficiently large, we arrive at \eqref{est-R-q+1}, \eqref{est-w_q-L^2_x} and \eqref{est-w_q-W^1_x}. This completes the proof of Lemma \ref{Lemma:Iteration}.

\vskip 0.125in

{\bf Acknowledgement}
The authors would like to thank H. Ibdah and the anonymous referee for carefully reading the paper and for their constructive suggestions.
The authors would also like to thank the ``The Institute of Mathematical Sciences", Chinese University of Hong Kong, for the warm and kind hospitality during which part of this work was completed.   The work of T.L. is supported in part by NSFC Grants 11601258. The work of E.S.T.  is supported in part by the ONR grant N00014-15-1-2333, the Einstein Stiftung/Foundation - Berlin, through the Einstein Visiting Fellow Program, and by the John Simon Guggenheim Memorial Foundation.

\end{document}